\documentclass[11pt]{csdm}

\usepackage{url,floatflt,helvet,times}
\usepackage{psfig,graphics}
\usepackage{mathptmx,amssymb,amsmath,bm}
\usepackage{float,tcolorbox,xcolor}
\usepackage[bf,hypcap]{caption}
\usepackage{graphicx,wrapfig,subfig}
\usepackage{setspace,enumerate,bookmark,placeins}
\usepackage{multirow,multicol}
\usepackage[curve,frame,line,arrow,matrix]{xy}
\usepackage[latin5]{inputenc}
\usepackage{tikz,calc}
\usetikzlibrary{arrows,backgrounds,automata}
\usetikzlibrary{decorations.markings}
\tikzstyle{vertex}=[circle, draw, inner sep=2pt, minimum size=6pt]
\newcommand{\vertex}{\node[vertex]}

\emergencystretch=\maxdimen
\allowdisplaybreaks

\newcommand{\E}{\mu}
\newcommand{\C}{\mathcal{C}}
\newcommand{\V}{\sigma^2}
\newcommand{\eE}{\E_{\chi_e}}
\newcommand{\eV}{\V_{\chi_e}}

\def\firstpage{3}
\def\thevol{1}
\def\thenumber{1}
\def\theyear{2017}

\begin{document}
{\fontfamily{cmss}\selectfont

\titlefigurecaption{\hspace{0.3cm}{\hspace{0.6cm}\LARGE \bf \sc \sffamily\color{white} Contemporary Studies in Discrete Mathematics}}

\title{\sc \sffamily On Equitable Coloring Parameters of Certain Wheel Related Graphs}

\author{\sc\sffamily K.P. Chithra$^{1,\ast}$, E.A. Shiny$^2$ and N.K. Sudev$^3$}
\institute {$^{1}$ Naduvath Mana, Nandikkara, Thrissur, India.\\
$^{2}$Centre for Studies in Discrete Mathematics, Vidya Academy of Science \& Technology, Thalakkottukara,Thrissur, India.\\ 
$^{3}$Centre for Studies in Discrete Mathematics, Vidya Academy of Science \& Technology,Thalakkottukara, Thrissur, India.
}


\titlerunning{Equitable Coloring Parameters of Wheel Related Graphs}
\authorrunning{K.P. Chithra, E.A. Shiny \& N.K. Sudev }

\mail{chithrasudev@gmail.com}

\received{15 May 2017}
\revised{24 July 2017}
\accepted{17 August 2017}
\published{20 October 2017.}

\abstracttext{Coloring the vertices of a graph $G$ subject to given conditions can be considered as a random experiment and corresponding to this experiment, a discrete random variable $X$ can be defined as the colour of a vertex chosen at random, with respect to the given type of coloring of $G$ and a probability mass function for this random variable can be defined accordingly. A proper coloring $\C$ of a graph $G$, which assigns colors to the vertices of $G$ such that the numbers of vertices in any two color classes differ by at most one, is called an equitable colouring of $G$.  In this paper, we study two statistical parameters of certain cycle related graphs, with respect to their equitable colorings.}

\maketitle

\setstretch{1.30}

\section{Introduction}

For all  terms and definitions, not defined specifically in this paper, we refer to \cite{BM1,FH,DBW} and for the terminology of graph coloring, we refer to \cite{CZ1,JT1,MK1}.  For the terminology in Statistics, see \cite{VS1,SMR1}. Unless mentioned otherwise, all graphs considered in this paper are simple, finite, connected and undirected.

\textit{Graph coloring} is an assignment of colors or labels or weights to the vertices, edges and faces of a graph under consideration. Many practical and real life situations motivated in the development of different types of graph coloring problems. Unless stated otherwise, the graph coloring is meant to be an assignment of colors to the vertices of a graph subject to certain conditions. A \textit{proper coloring} of a graph $G$ is a colouring with respect to which vertices of $G$ are colored in such a way that no two adjacent vertices $G$ have the same color. 

Several parameters related to different types of graph coloring can be seen in the literature concerned. The \textit{chromatic number} of graphs is the minimum number of colors required in a proper coloring of the given graph. The general coloring sums with respect to different classes have been studied in \cite{KSC1,SCK1}. 

Coloring of the vertices of a given graph $G$ can be considered as a random experiment. For a proper $k$-coloring $\C = \{c_1,c_2, c_3, \ldots,c_k\}$ of $G$, we can define a random variable (\textit{r.v.}) $X$ which denotes the color ( or precisely, the subscript $i$ of the color $c_i$) of any arbitrary vertex in $G$. As the sum of all weights of colors of $G$ is equal to the number of vertices of $G$, the real valued function $f(i)$ defined by 
$$f(i)= 
\begin{cases}
\frac{\theta(c_i)}{|V(G)|}; &  i=1,2,3,\ldots,k\\
0; & \text{elsewhere}
\end{cases}$$ 
will be the \textit{probability mass function (p.m f.)}  of the random variable (\textit{r.v.}) $X$ (see \cite{SSCK1}), where $\theta(c_i)$ denotes the cardinality of the colour class of the colour $c_i$. 

If the context is clear, we can also say that $f(i)$ is the probability mass function of the graph $G$ with respect to the given coloring $\C$. Hence, we can also define the parameters like mean and variance of the random variable $X$, with respect to a general coloring of $G$ can be defined as follows.

An {\em equitable coloring} of a graph $G$ is a proper coloring $\C$  of $G$  which an assignment of colors to the vertices of $G$ such that the numbers of vertices in any two color classes differ by at most one. The \textit{equitable chromatic number} of a graph $G$ is the smallest number $k$ such that $G$ has an equitable coloring with $k$ colors. 

\begin{conjecture}[Equitable Coloring Conjecture (ECC)]
\textup{\cite{WM1}} For any connected graph $G$, which is neither a complete graph nor an odd cycle,  $\chi_e(G)\le \Delta(G)$, where $\Delta(G)$ is the maximum vertex degree in $G$.
\end{conjecture}

In \cite{LW1}, this conjecture was verified for connected bipartite graphs and following it, a study on the relation between the equitable chromatic number and the maximum degree of graphs in \cite{CLW} and a subsequent study on equitable coloring of trees has been done in \cite{CL1}. 

Motivated by the studies on different types graph colorings, coloring parameters and their applications, we extend the concepts of arithmetic mean and variance, two statistical parameters, to the theory of equitable graph coloring in this paper. Throughout the paper, we follow the convention that $0\le \theta(c_i)-\theta(c_j)\le 1$, when $i<j$.

\section{New Results}

Two types of chromatic means and variances corresponding to an equitable coloring of a graph $G$ are defined as follows.

\begin{definition}\textup{\cite{SCSK3}
Let $\C = \{c_1,c_2, c_3, \ldots,c_k\}$ be an equitable $k$-coloring of a given graph $G$ and $X$ be the random variable which denotes the number of vertices having a particular color in $\C$, with the $p.m.f$ $f(i)$.  Then, 
\begin{enumerate}\itemsep0mm
\item[(i)] the \textit{equitable coloring mean} of a coloring $\C$ of a given graph $G$, denoted by $\eE(G)$, is defined to be $\eE(G) =\sum\limits_{i=1}^{n}i\,f(i)$;
\item[(ii)]  \textit{equitable coloring variance} of a coloring $\C$ of a given graph $G$, denoted by $\eV(G)$, is defined to be $\eV(G) =\sum\limits_{i=1}^{n}i^2\,f(i)-\left(\sum\limits_{i=1}^{n}i\,f(i)\right)^2$.
\end{enumerate}}
\end{definition}

If the context is clear, the above defined parameters can respectively be called the equitable coloring mean and equitable coloring variance of the graph $G$ with respect to the coloring $\C$. In this paper, we study the equitable chromatic parameters for certain cycle related graph classes. 

A \textit{wheel graph} $W_n$ is a graph obtained by joining all vertices of a cycle $C_n$ to an external vertex. This external vertex may be called the \textit{central vertex} of $W_n$ and the cycle $C_n$ may be called the \textit{rim} of $W_n$. The following result determines the equitable colouring parameters of a wheel graph.

\begin{theorem}\label{Thm-1}
For a wheel graph $W_n=C_n+K_1$, we have
$$\eE(W_n)=
\begin{cases}
\frac{(n+2)^2}{4(n+1)}; & \text{if $n$ is even}\\
\frac{(n^2+4n+7)}{4(n+1)}; & \text{if $n$ is odd},
\end{cases}$$ 
and
$$\eV(W_n)=\begin{cases}
\frac{n(n+2)(n^2+2n+4)}{48(n+1)^2}; & \text{if $n$ is even}\\
\frac{n^4+4n^3+26n^2-44n-27}{48(n+1)^2}; & \text{if $n$ is odd}.
\end{cases}$$
\end{theorem}
\begin{proof}~
Let $v$ be the central vertex of the wheel graph $W_n=C_n+K_1$ and $v_1,v_2,\ldots,v_n$ be the vertices of the cycle $C_n$. Since $v$ is adjacent to all other vertices of $W_n$, none of them can have the same colour of $v$. Hence, the colour class containing $v$ is a singleton set. Hence, every other colour class, with respect to an equitable colouring of $W_n$ must be a singleton or a $2$-element set. Here we need to consider the following cases.

\textit{Case-1:} Let $n$ be even. Then, $\frac{n}{2}+1$ colours are required in an equitable colouring of $W_n$. Let $\C=\{c_1,c_2,\ldots,c_\ell,c_{\ell+1}\}$ be an equitable colouring of $W_n$, where $\ell=\frac{n}{2}$. Note that the colour classes of $c_1,c_2,\ldots,c_\ell$ are $2$-element sets, while the colour class of $c_{\ell+1}$, which consists of the central vertex, is a singleton set (see Figure \ref{fig:e-wl} for illustration). Let $X$ be the random variable which represents the colour of an arbitrarily chosen vertex of $W_n$. Then, the corresponding \textit{p.m.f} of $G$ is 
$$f(i)=P(X=i)=
\begin{cases}
\frac{2}{n+1}; & \text{for}\ i=1,2,\ldots, \frac{n}{2}\\
\frac{1}{n+1}; & \text{for}\ i=\frac{n}{2}+1\\
0; & \text{elsewhere}.
\end{cases}$$
Hence, the corresponding equitable colouring mean of the wheel graph $W_n$ is given by $\eE(W_n)=\left(1+2+3+\ldots+\frac{n}{2}\right)\cdot \frac{2}{n+1} +\left(\frac{n}{2}+1\right)\cdot \frac{1}{n+1}=\frac{(n+2)^2}{4(n+1)}$ and the corresponding equitable colouring variance of $W_n$ is $\eV(W_n)= \left(1^2+2^2+3^2+\ldots+\frac{n^2}{4}\right)\cdot \frac{2}{n+1} +\left(\frac{n}{2}+1\right)^2\cdot \frac{1}{n+1}-\left(\frac{(n+2)^2}{4(n+1)}\right)^2=\frac{n(n+2)(n^2+2n+4)}{48(n+1)^2}$.

\textit{Case-2:} Let $n$ be odd. Then, $\frac{n-1}{2}+2$ colours are required in an equitable colouring of $W_n$. If $\C=\{c_1,c_2,\ldots,c_\ell,c_{\ell+1},c_{\ell+2}\}$ be an equitable colouring of $W_n$, where $\ell=\frac{n-1}{2}$, then the colour classes of $c_1,c_2,\ldots,c_\ell$ are $2$-element sets, while the colour classes of $c_{\ell+1}$ and $c_{\ell+2}$ are singleton sets (see Figure \ref{fig:e-wl} for illustration). Let $X$ be the random variable which represents the colour of an arbitrarily chosen vertex of $W_n$. Then, the corresponding \textit{p.m.f} of $G$ is 
$$f(i)=P(X=i)=
\begin{cases}
\frac{2}{n+1}; & \text{for}\ i=1,2,\ldots, \frac{n-1}{2}\\
\frac{1}{n+1}; & \text{for}\ i=\frac{n-1}{2}+1,\frac{n-1}{2}+2\\
0; & \text{elsewhere}.
\end{cases}$$
Then, the corresponding equitable colouring mean is $\eE=\left(1+2+3+\ldots+\frac{n-1}{2}\right)\cdot \frac{2}{n+1} +\left(\frac{n-1}{2}+1+\frac{n-1}{2}+2\right)\cdot \frac{1}{n+1}=\frac{(n^2+4n+7)}{4(n+1)}$ and $\eV= \left(1^2+2^2+3^2+\ldots+\frac{(n-1)^2}{4}\right)\cdot \frac{2}{n+1} +\left[\left(\frac{n-1}{2}+1\right)^2+\left(\frac{n-1}{2}+2\right)^2\right]\cdot \frac{1}{n+1}-\left(\frac{n^+4n+7}{4(n+1)}\right)^2 =\frac{n^4+76n^3+386n^2+692n+333}{48(n+1)^2}$.
\end{proof}

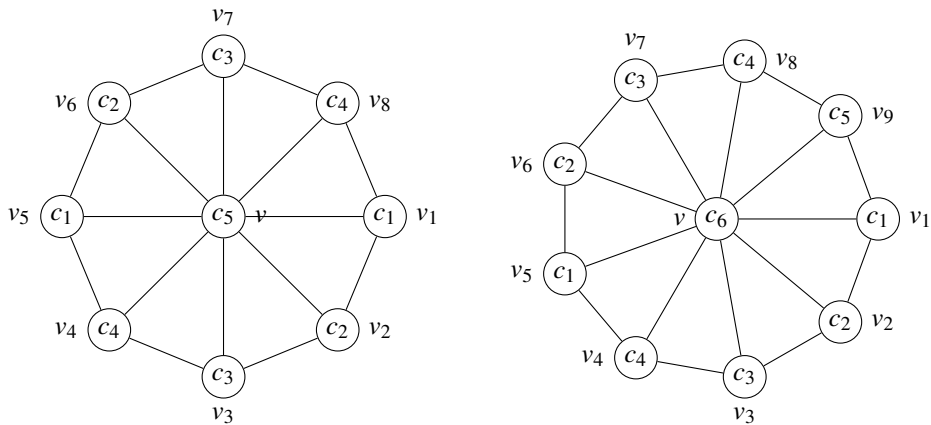
\begin{figure}[h!]
\begin{center}
\begin{tikzpicture}[scale=0.85] 
\vertex (v1) at (0:2.5) [label=right:$v_{1}$]{$c_1$};
\vertex (v2) at (315:2.5) [label=right:$v_{2}$]{$c_2$};
\vertex (v3) at (270:2.5) [label=below:$v_{3}$]{$c_3$};
\vertex (v4) at (225:2.5) [label=left:$v_{4}$]{$c_4$};
\vertex (v5) at (180:2.5) [label=left:$v_{5}$]{$c_1$};
\vertex (v6) at (135:2.5) [label=left:$v_{6}$]{$c_2$};
\vertex (v7) at (90:2.5) [label=above:$v_{7}$]{$c_3$};
\vertex (v8) at (45:2.5) [label=right:$v_{8}$]{$c_4$};
\vertex (v) at (0:0) [label=right:$v$]{$c_5$};
\path 
(v1) edge (v2)
(v1) edge (v8)
(v1) edge (v)
(v2) edge (v3)
(v2) edge (v)
(v3) edge (v4)
(v3) edge (v)
(v4) edge (v5)
(v4) edge (v)
(v5) edge (v6)
(v5) edge (v)
(v6) edge (v7)
(v6) edge (v)
(v7) edge (v8)
(v7) edge (v)
(v8) edge (v)
;
\end{tikzpicture}
\qquad 
\begin{tikzpicture}[scale=0.85] 
\vertex (v1) at (0:2.5) [label=right:$v_{1}$]{$c_1$};
\vertex (v2) at (320:2.5) [label=right:$v_{2}$]{$c_2$};
\vertex (v3) at (280:2.5) [label=below:$v_{3}$]{$c_3$};
\vertex (v4) at (240:2.5) [label=left:$v_{4}$]{$c_4$};
\vertex (v5) at (200:2.5) [label=left:$v_{5}$]{$c_1$};
\vertex (v6) at (160:2.5) [label=left:$v_{6}$]{$c_2$};
\vertex (v7) at (120:2.5) [label=above:$v_{7}$]{$c_3$};
\vertex (v8) at (80:2.5) [label=right:$v_{8}$]{$c_4$};
\vertex (v9) at (40:2.5) [label=right:$v_{9}$]{$c_5$};
\vertex (v) at (0:0) [label=left:$v$]{$c_6$};
\path 
(v1) edge (v2)
(v2) edge (v3)
(v3) edge (v4)
(v4) edge (v5)
(v5) edge (v6)
(v6) edge (v7)
(v7) edge (v8)
(v8) edge (v9)
(v9) edge (v1)
(v1) edge (v)
(v2) edge (v)
(v3) edge (v)
(v4) edge (v)
(v5) edge (v)
(v6) edge (v)
(v7) edge (v)
(v8) edge (v)
(v9) edge (v)
;
\end{tikzpicture}
\end{center}
\caption{\small Equitable colouring of wheel graphs}\label{fig:e-wl}
\end{figure}

A \textit{double wheel graph} $DW_n$ is a graph defined by $2C_n+K_1$. That is, a double wheel graph is a graph obtained by joining all vertices of the two disjoint cycles to an external vertex. The following result discusses the equitable colouring parameters of a double wheel graph.

\begin{theorem}\label{Thm-1a}
For a double wheel graph $DW_n=2C_n+K_1$, we have
$$\eE(DW_n)= \frac{(n+1)^2}{2n+1}$$ 
and
$$\eV(DW_n)=\frac{n^4+2n^3+2n^2+n}{3(2n+1)^2}.$$
\end{theorem}
\begin{proof}~
Note that there are $n+1$ colours, say $c_1,c_2,\ldots,c_n,c_{n+1}$, in an equitable colouring of $DW_n$. As in the case of wheel graph, the colour classes of $c_1,c_2,\ldots,c_n$ are $2$-element sets and the colour class of $c_{n+1}$ is a singleton set in $DW_n$ (see Figure \ref{fig:e-dwl} for illustration). Therefore, the corresponding \textit{p.m.f} of $DW_n$ is given by
$$f(i)=P(X=i)=
\begin{cases}
\frac{2}{2n+1}; & \text{if}\ i=1,2,3,\ldots,n;\\
\frac{1}{2n+1}; & \text{if}\ i=n+1;\\
0; & \text{elsewhere}.
\end{cases}$$
Hence, $\eE(DW_n)=(1+2+3+\ldots+n)\cdot\frac{2}{2n+1}+(n+1)\cdot\frac{1}{2n+1}=\frac{(n+1)^2}{2n+1}$ and $\eV(DW_n)= (1^2+2^2+\ldots+n^2)\cdot\frac{2}{2n+1}+(n+1)^2\cdot\frac{1}{2n+1}-\left(\frac{(n+1)^2}{2n+1}\right)^2 = \frac{n^4+2n^3+2n^2+n}{3(2n+1)^2}$.
\end{proof}

\begin{figure}[h!]
\centering
\begin{tikzpicture}[auto,node distance=1.75cm,
thick,main node/.style={circle,draw,font=\sffamily\Large\bfseries}]
\vertex (v1) at (0:1.75) []{$c_1$};
\vertex (v2) at (320:1.75) []{$c_2$};
\vertex (v3) at (280:1.75) []{$c_3$};
\vertex (v4) at (240:1.75) []{$c_4$};
\vertex (v5) at (200:1.75) []{$c_5$};
\vertex (v6) at (160:1.75) []{$c_6$};
\vertex (v7) at (120:1.75) []{$c_7$};
\vertex (v8) at (80:1.75) []{$c_8$};
\vertex (v9) at (40:1.75) []{$c_9$};
\vertex (v) at (0:0) []{$c_{10}$};
\vertex (u1) at (0:3) []{$c_9$};
\vertex (u2) at (320:3) []{$c_8$};
\vertex (u3) at (280:3) []{$c_7$};
\vertex (u4) at (240:3) []{$c_6$};
\vertex (u5) at (200:3) []{$c_4$};
\vertex (u6) at (160:3) []{$c_5$};
\vertex (u7) at (120:3) []{$c_3$};
\vertex (u8) at (80:3) []{$c_2$};
\vertex (u9) at (40:3) []{$c_1$};
\path 
(v1) edge (v2)
(v2) edge (v3)
(v3) edge (v4)
(v4) edge (v5)
(v5) edge (v6)
(v6) edge (v7)
(v7) edge (v8)
(v8) edge (v9)
(v9) edge (v1)
(u1) edge (u2)
(u2) edge (u3)
(u3) edge (u4)
(u4) edge (u5)
(u5) edge (u6)
(u6) edge (u7)
(u7) edge (u8)
(u8) edge (u9)
(u9) edge (u1)
(v1) edge (v)
(v2) edge (v)
(v3) edge (v)
(v4) edge (v)
(v5) edge (v)
(v6) edge (v)
(v7) edge (v)
(v8) edge (v)
(v9) edge (v)
(u1) edge [bend right] (v)
(u2) edge [bend right] (v)
(u3) edge [bend right] (v)
(u4) edge [bend right] (v)
(u5) edge [bend right] (v)
(u6) edge [bend right] (v)
(u7) edge [bend right] (v)
(u8) edge [bend right] (v)
(u9) edge [bend right] (v)
;
\end{tikzpicture}
\caption{\small Equitable colouring of a double wheel graph}\label{fig:e-dwl}
\end{figure}
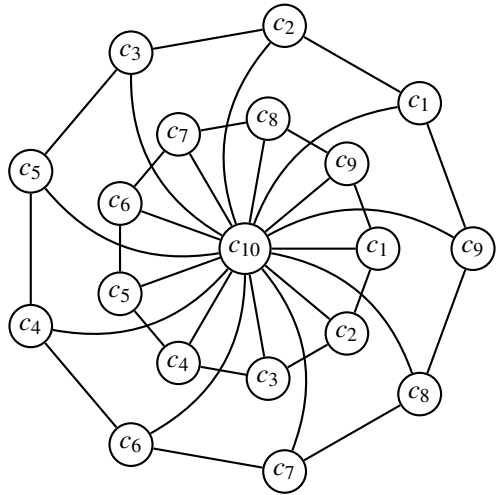 

A \textit{helm graph} $H_n$ is a graph obtained by attaching a pendant edge to every vertex of the rim $C_n$ of a wheel graph $W_n$. The following result provides the equitable colouring parameters of the helm graphs.

\begin{theorem}\label{Thm-2}
For the helm graph $H_n$, we have $$\eE(H_n)=\frac{5n+1}{2n+1}$$ 
and
$$\eV(H_n)=\begin{cases}
\frac{5n^2+7n}{(2n+1)^2}; & \text{if $n$ is even}\\
\frac{5n^2+3n-2}{(2n+1)^2}; & \text{if $n$ is odd}.
\end{cases}$$
\end{theorem}
\begin{proof}~
Note that the equitable colouring of a helm graph $H_n$ contains four colours, say $c_1,c_2,c_3$ and $c_4$. Then, we consider the following cases:

\textit{Case-1:} If $n$ is even, then $\theta(c_1)=\frac{n+2}{2}$ and $\theta(c_2)=\theta(c_3)=\theta(c_4)=\frac{n}{2}$ (see Figure \ref{fig:e-hl} for illustration). Then, the corresponding to the random variable $X$ defined as mentioned above, the \textit{p.m.f} is given by 
$$f(i)=P(X=i)=
\begin{cases}
\frac{n+2}{2(2n+1)}; & \text{for}\ i=1,\\
\frac{n}{2(2n+1)}; & \text{for}\ i=2,3,4\\
0; & \text{elsewhere}.
\end{cases}$$
Then, the corresponding equitable colouring mean is $\eE=1\cdot \frac{n+2}{2(2n+1)} +(2+3+4)\cdot \frac{n}{2(2n+1)}=\frac{5n+1}{2n+1}$ and $\eV= 1^2\cdot \frac{n+2}{4n+2} +(2^2+3^2+4^2)\cdot \frac{n}{4n+2}-\left(\frac{5n+1}{4n+2}\right)^2 =\frac{5n^2+7n}{(2n+1)^2}$.

\textit{Case-2:} If $n$ is odd, then $\theta(c_1)=\theta(c_2)=\theta(c_3)=\frac{n+1}{2}$ and $\theta(c_4)=\frac{n-1}{2}$ (see Figure \ref{fig:e-hl} for illustration). Then, the corresponding to the random variable $X$ defined as mentioned above, the \textit{p.m.f} is given by 
$$f(i)=P(X=i)=
\begin{cases}
\frac{n+1}{2(2n+1)}; & \text{for}\ i=1,2,3\\
\frac{n-1}{2(2n+1)}; & \text{for}\ i=4\\
0; & \text{elsewhere}.
\end{cases}$$
Then, the corresponding equitable colouring mean is $\eE=(1+2+3)\cdot \frac{n+1}{2(2n+1)} +4\cdot \frac{n-1}{2(2n+1)}=\frac{5n+1}{2n+1}$ and $\eV= (1^2+2^2+3^2)\cdot \frac{n+1}{4n+2}+4^2\cdot \frac{n-1}{4n+2}-\left(\frac{5n+1}{4n+2}\right)^2 =\frac{5n^2+3n-2}{(2n+1)^2}$.
\end{proof}

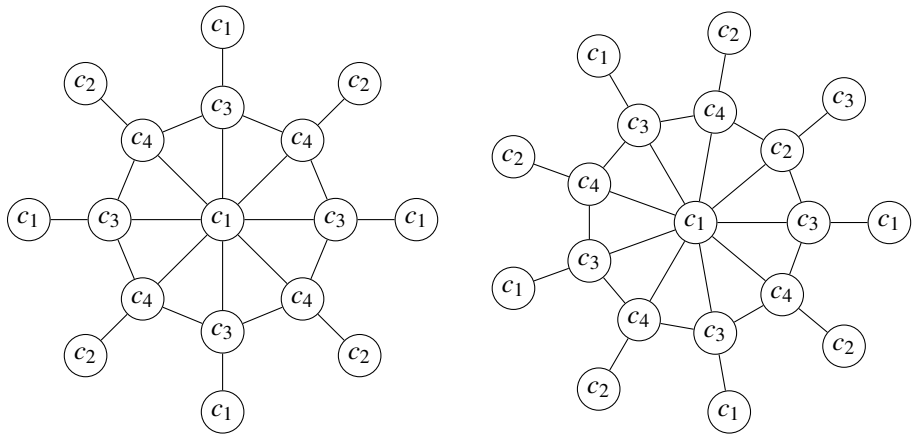
\begin{figure}[h!]
\begin{center}
\begin{tikzpicture}[scale=0.85] 
\vertex (v1) at (0:1.75) []{$c_3$};
\vertex (v2) at (315:1.75) []{$c_4$};
\vertex (v3) at (270:1.75) []{$c_3$};
\vertex (v4) at (225:1.75) []{$c_4$};
\vertex (v5) at (180:1.75) []{$c_3$};
\vertex (v6) at (135:1.75) []{$c_4$};
\vertex (v7) at (90:1.75) []{$c_3$};
\vertex (v8) at (45:1.75) []{$c_4$};
\vertex (v) at (0:0) []{$c_1$};
\vertex (u1) at (0:3) []{$c_1$};
\vertex (u2) at (315:3) []{$c_2$};
\vertex (u3) at (270:3) []{$c_1$};
\vertex (u4) at (225:3) []{$c_2$};
\vertex (u5) at (180:3) []{$c_1$};
\vertex (u6) at (135:3) []{$c_2$};
\vertex (u7) at (90:3) []{$c_1$};
\vertex (u8) at (45:3) []{$c_2$};
\path 
(v1) edge (v2)
(v1) edge (v8)
(v1) edge (v)
(v2) edge (v3)
(v2) edge (v)
(v3) edge (v4)
(v3) edge (v)
(v4) edge (v5)
(v4) edge (v)
(v5) edge (v6)
(v5) edge (v)
(v6) edge (v7)
(v6) edge (v)
(v7) edge (v8)
(v7) edge (v)
(v8) edge (v)
(v1) edge (u1)
(v2) edge (u2)
(v3) edge (u3)
(v4) edge (u4)
(v5) edge (u5)
(v6) edge (u6)
(v7) edge (u7)
(v8) edge (u8)
;
\end{tikzpicture}
\qquad 
\begin{tikzpicture}[scale=0.85] 
\vertex (v1) at (0:1.75) []{$c_3$};
\vertex (v2) at (320:1.75) []{$c_4$};
\vertex (v3) at (280:1.75) []{$c_3$};
\vertex (v4) at (240:1.75) []{$c_4$};
\vertex (v5) at (200:1.75) []{$c_3$};
\vertex (v6) at (160:1.75) []{$c_4$};
\vertex (v7) at (120:1.75) []{$c_3$};
\vertex (v8) at (80:1.75) []{$c_4$};
\vertex (v9) at (40:1.75) []{$c_2$};
\vertex (v) at (0:0) []{$c_1$};
\vertex (u1) at (0:3) []{$c_1$};
\vertex (u2) at (320:3) []{$c_2$};
\vertex (u3) at (280:3) []{$c_1$};
\vertex (u4) at (240:3) []{$c_2$};
\vertex (u5) at (200:3) []{$c_1$};
\vertex (u6) at (160:3) []{$c_2$};
\vertex (u7) at (120:3) []{$c_1$};
\vertex (u8) at (80:3) []{$c_2$};
\vertex (u9) at (40:3) []{$c_3$};
\path 
(v1) edge (v2)
(v2) edge (v3)
(v3) edge (v4)
(v4) edge (v5)
(v5) edge (v6)
(v6) edge (v7)
(v7) edge (v8)
(v8) edge (v9)
(v9) edge (v1)
(v1) edge (v)
(v2) edge (v)
(v3) edge (v)
(v4) edge (v)
(v5) edge (v)
(v6) edge (v)
(v7) edge (v)
(v8) edge (v)
(v9) edge (v)
(v1) edge (u1)
(v2) edge (u2)
(v3) edge (u3)
(v4) edge (u4)
(v5) edge (u5)
(v6) edge (u6)
(v7) edge (u7)
(v8) edge (u8)
(v9) edge (u9)
;
\end{tikzpicture}
\end{center}
\caption{\small Equitable colouring of helm graphs}\label{fig:e-hl}
\end{figure}

A \textit{closed helm graph} $CH_n$ is a graph obtained from the helm graph $H_n$, by joining a pendant vertex $v_i$ to the pendant vertex $v_{i+1}$, where $1\le i\le n$ and $v_{n+i}=v_i$. That is, the pendant vertices in $H_n$ induce a cycle in $CH_n$. Then, we have 

\begin{theorem}\label{Thm-3}
For the closed helm graph $CH_n$, we have $$\eE(CH_n)=\frac{5n+1}{2n+1}$$ 
and
$$\eV(CH_n)=
\begin{cases}
\frac{5n^2+7n}{(2n+1)^2}; & \text{if $n$ is even}\\
\frac{5n^2+3n-2}{(2n+1)^2}; & \text{if $n$ is odd}.
\end{cases}$$
\end{theorem}
\begin{proof}~
The proof follows exactly as mentioned in the proof Theorem \ref{Thm-2} (see Figure \ref{fig:e-chl} for illustration).
\end{proof}

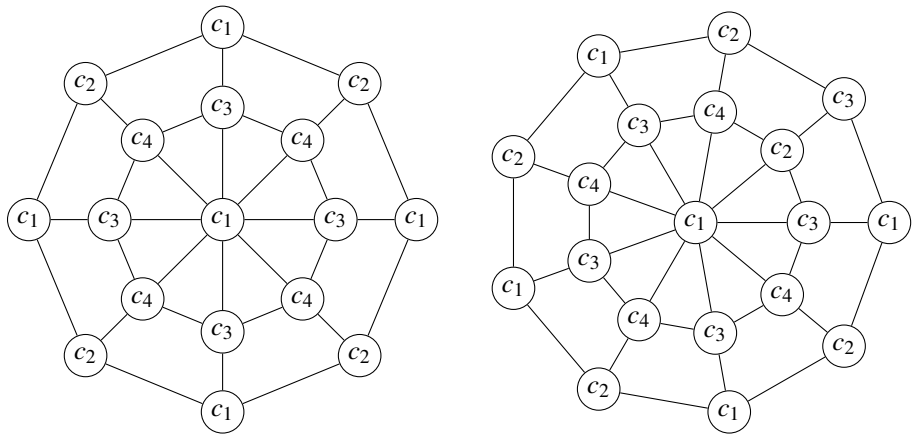
\begin{figure}[h!]
\begin{center}
\begin{tikzpicture}[scale=0.85] 
\vertex (v1) at (0:1.75) []{$c_3$};
\vertex (v2) at (315:1.75) []{$c_4$};
\vertex (v3) at (270:1.75) []{$c_3$};
\vertex (v4) at (225:1.75) []{$c_4$};
\vertex (v5) at (180:1.75) []{$c_3$};
\vertex (v6) at (135:1.75) []{$c_4$};
\vertex (v7) at (90:1.75) []{$c_3$};
\vertex (v8) at (45:1.75) []{$c_4$};
\vertex (v) at (0:0) []{$c_1$};
\vertex (u1) at (0:3) []{$c_1$};
\vertex (u2) at (315:3) []{$c_2$};
\vertex (u3) at (270:3) []{$c_1$};
\vertex (u4) at (225:3) []{$c_2$};
\vertex (u5) at (180:3) []{$c_1$};
\vertex (u6) at (135:3) []{$c_2$};
\vertex (u7) at (90:3) []{$c_1$};
\vertex (u8) at (45:3) []{$c_2$};
\path 
(v1) edge (v2)
(v1) edge (v8)
(v1) edge (v)
(v2) edge (v3)
(v2) edge (v)
(v3) edge (v4)
(v3) edge (v)
(v4) edge (v5)
(v4) edge (v)
(v5) edge (v6)
(v5) edge (v)
(v6) edge (v7)
(v6) edge (v)
(v7) edge (v8)
(v7) edge (v)
(v8) edge (v)
(v1) edge (u1)
(v2) edge (u2)
(v3) edge (u3)
(v4) edge (u4)
(v5) edge (u5)
(v6) edge (u6)
(v7) edge (u7)
(v8) edge (u8)
(u1) edge (u2)
(u2) edge (u3)
(u3) edge (u4)
(u4) edge (u5)
(u5) edge (u6)
(u6) edge (u7)
(u7) edge (u8)
(u8) edge (u1)
;
\end{tikzpicture}
\qquad 
\begin{tikzpicture}[scale=0.85] 
\vertex (v1) at (0:1.75) []{$c_3$};
\vertex (v2) at (320:1.75) []{$c_4$};
\vertex (v3) at (280:1.75) []{$c_3$};
\vertex (v4) at (240:1.75) []{$c_4$};
\vertex (v5) at (200:1.75) []{$c_3$};
\vertex (v6) at (160:1.75) []{$c_4$};
\vertex (v7) at (120:1.75) []{$c_3$};
\vertex (v8) at (80:1.75) []{$c_4$};
\vertex (v9) at (40:1.75) []{$c_2$};
\vertex (v) at (0:0) []{$c_1$};
\vertex (u1) at (0:3) []{$c_1$};
\vertex (u2) at (320:3) []{$c_2$};
\vertex (u3) at (280:3) []{$c_1$};
\vertex (u4) at (240:3) []{$c_2$};
\vertex (u5) at (200:3) []{$c_1$};
\vertex (u6) at (160:3) []{$c_2$};
\vertex (u7) at (120:3) []{$c_1$};
\vertex (u8) at (80:3) []{$c_2$};
\vertex (u9) at (40:3) []{$c_3$};
\path 
(v1) edge (v2)
(v2) edge (v3)
(v3) edge (v4)
(v4) edge (v5)
(v5) edge (v6)
(v6) edge (v7)
(v7) edge (v8)
(v8) edge (v9)
(v9) edge (v1)
(v1) edge (v)
(v2) edge (v)
(v3) edge (v)
(v4) edge (v)
(v5) edge (v)
(v6) edge (v)
(v7) edge (v)
(v8) edge (v)
(v9) edge (v)
(v1) edge (u1)
(v2) edge (u2)
(v3) edge (u3)
(v4) edge (u4)
(v5) edge (u5)
(v6) edge (u6)
(v7) edge (u7)
(v8) edge (u8)
(v9) edge (u9)
(u1) edge (u2)
(u2) edge (u3)
(u3) edge (u4)
(u4) edge (u5)
(u5) edge (u6)
(u6) edge (u7)
(u7) edge (u8)
(u8) edge (u9)
(u9) edge (u1)
;
\end{tikzpicture}
\end{center}
\caption{\small Equitable colouring of closed helm graphs}\label{fig:e-chl}
\end{figure}

A \textit{flower graph} $F_n$ is a graph which is obtained by joining the pendant vertices of a helm graph $H_n$ to its central vertex. The following theorem discusses the equitable colouring parameters of the flower graph $F_n$.

\begin{theorem}\label{Thm-4}
For a flower graph $F_n$, we have $$\eE(F_n)=\frac{(n+1)^2}{2n+1}$$
and $$\eV(F_n)=\frac{n^4+2n^3+2n^2+n}{3(2n+1)^2}.$$
\end{theorem}
\begin{proof}~
As in the case of wheel graph, the central vertex $v$ of $F_n$ is adjacent to all other vertices of $F_n$ and hence no other vertex of $F_n$ can have the same colour of $v$. Hence, the colour class containing $v$ is a singleton set and every other colour class, with respect to an equitable colouring of $F_n$ must be a singleton set or a $2$-element set. Then, $n+1$ colours are required in an equitable colouring of $F_n$. Let $\C=\{c_1,c_2,\ldots,c_n,c_{n+1}\}$ be an equitable colouring of $F_n$. Note that the colour classes of $c_1,c_2,\ldots,c_n$ are $2$-element sets, while the colour class of $c_{n+1}$, which consists of the central vertex, is a singleton set (see Figure \ref{fig:e-fl} for illustration). Then, the corresponding \textit{p.m.f} of $G$ is 
$$f(i)=P(X=i)=
\begin{cases}
\frac{2}{2n+1}; & \text{for}\ i=1,2,\ldots,n,\\
\frac{1}{2n+1}; & \text{for}\ i=n+1\\
0; & \text{elsewhere}.
\end{cases}$$
Hence, the corresponding equitable colouring mean of the flower graph $F_n$ is given by $\eE(F_n)=(1+2+3+\ldots+n)\cdot \frac{2}{2n+1} +(n+1)\cdot \frac{1}{2n+1}=\frac{(n+1)^2}{2n+1}$ and the corresponding equitable colouring variance of $H_n$ is $\eV= (1^2+2^2+3^2+\ldots+n^2)\cdot \frac{2}{2n+1} +(n+1)^2\cdot \frac{1}{2n+1}-\left(\frac{(n+1)^2}{2n+1}\right)^2=\frac{n^4+2n^3+2n^2+n}{3(2n+1)^2}$.
\end{proof}

\begin{figure}[h!]
\begin{center}
\begin{tikzpicture}[scale=0.85] 
\vertex (v1) at (0:1.75) []{$c_1$};
\vertex (v2) at (315:1.75) []{$c_2$};
\vertex (v3) at (270:1.75) []{$c_3$};
\vertex (v4) at (225:1.75) []{$c_4$};
\vertex (v5) at (180:1.75) []{$c_5$};
\vertex (v6) at (135:1.75) []{$c_6$};
\vertex (v7) at (90:1.75) []{$c_7$};
\vertex (v8) at (45:1.75) []{$c_8$};
\vertex (v) at (0:0) []{$c_9$};
\vertex (u1) at (0:3) []{$c_8$};
\vertex (u2) at (315:3) []{$c_7$};
\vertex (u3) at (270:3) []{$c_6$};
\vertex (u4) at (225:3) []{$c_5$};
\vertex (u5) at (180:3) []{$c_4$};
\vertex (u6) at (135:3) []{$c_3$};
\vertex (u7) at (90:3) []{$c_2$};
\vertex (u8) at (45:3) []{$c_1$};
\path 
(v1) edge (v2)
(v1) edge (v8)
(v1) edge (v)
(v2) edge (v3)
(v2) edge (v)
(v3) edge (v4)
(v3) edge (v)
(v4) edge (v5)
(v4) edge (v)
(v5) edge (v6)
(v5) edge (v)
(v6) edge (v7)
(v6) edge (v)
(v7) edge (v8)
(v7) edge (v)
(v8) edge (v)
(v1) edge (u1)
(v2) edge (u2)
(v3) edge (u3)
(v4) edge (u4)
(v5) edge (u5)
(v6) edge (u6)
(v7) edge (u7)
(v8) edge (u8)
(u1) edge [bend right] (v)
(u2) edge [bend right] (v)
(u3) edge [bend right] (v)
(u4) edge [bend right] (v)
(u5) edge [bend right] (v)
(u6) edge [bend right] (v)
(u7) edge [bend right] (v)
(u8) edge [bend right] (v)
;
\end{tikzpicture}
\qquad 
\begin{tikzpicture}[scale=0.85] 
\vertex (v1) at (0:1.75) []{$c_1$};
\vertex (v2) at (320:1.75) []{$c_2$};
\vertex (v3) at (280:1.75) []{$c_3$};
\vertex (v4) at (240:1.75) []{$c_4$};
\vertex (v5) at (200:1.75) []{$c_5$};
\vertex (v6) at (160:1.75) []{$c_6$};
\vertex (v7) at (120:1.75) []{$c_7$};
\vertex (v8) at (80:1.75) []{$c_8$};
\vertex (v9) at (40:1.75) []{$c_9$};
\vertex (v) at (0:0) []{$c_{10}$};
\vertex (u1) at (0:3) []{$c_9$};
\vertex (u2) at (320:3) []{$c_8$};
\vertex (u3) at (280:3) []{$c_7$};
\vertex (u4) at (240:3) []{$c_6$};
\vertex (u5) at (200:3) []{$c_4$};
\vertex (u6) at (160:3) []{$c_5$};
\vertex (u7) at (120:3) []{$c_3$};
\vertex (u8) at (80:3) []{$c_2$};
\vertex (u9) at (40:3) []{$c_1$};
\path 
(v1) edge (v2)
(v2) edge (v3)
(v3) edge (v4)
(v4) edge (v5)
(v5) edge (v6)
(v6) edge (v7)
(v7) edge (v8)
(v8) edge (v9)
(v9) edge (v1)
(v1) edge (v)
(v2) edge (v)
(v3) edge (v)
(v4) edge (v)
(v5) edge (v)
(v6) edge (v)
(v7) edge (v)
(v8) edge (v)
(v9) edge (v)
(v1) edge (u1)
(v2) edge (u2)
(v3) edge (u3)
(v4) edge (u4)
(v5) edge (u5)
(v6) edge (u6)
(v7) edge (u7)
(v8) edge (u8)
(v9) edge (u9)
(u1) edge [bend right] (v)
(u2) edge [bend right] (v)
(u3) edge [bend right] (v)
(u4) edge [bend right] (v)
(u5) edge [bend right] (v)
(u6) edge [bend right] (v)
(u7) edge [bend right] (v)
(u8) edge [bend right] (v)
(u9) edge [bend right] (v)
;
\end{tikzpicture}
\end{center}
\caption{\small Equitable colouring of flower graphs}\label{fig:e-fl}
\end{figure}
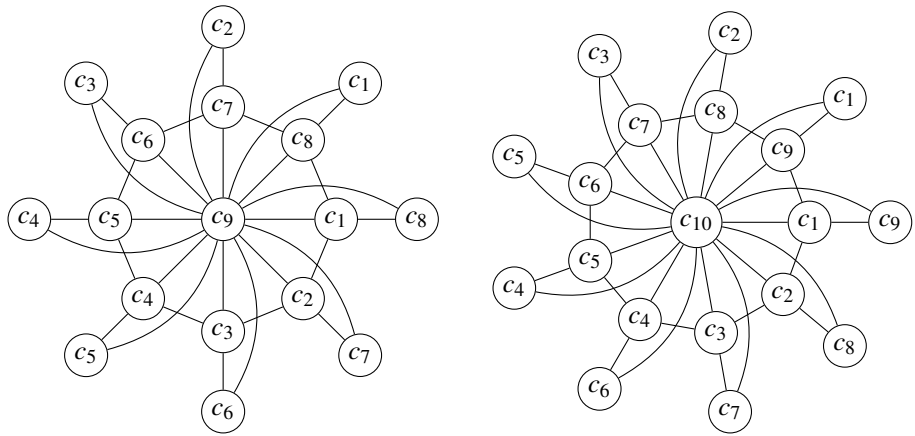

A \textit{sunflower graph} $SF_n$ is a graph obtained by replacing each edge of the rim of a wheel graph $W_n$ by a triangle such that two triangles share a common vertex if and only if the corresponding edges in $W_n$ are adjacent in $W_n$. The following theorem determines the equitable colouring parameters of the sunflower graph $SF_n$.

\begin{theorem}\label{Thm-5}
For the sunflower graph $SF_n$, we have $$\eE(SF_n)=\frac{5n+1}{2n+1}$$ 
and
$$\eV(SF_n)=\begin{cases}
\frac{5n^2+7n}{(2n+1)^2}; & \text{if $n$ is even}\\
\frac{5n^2+3n-2}{(2n+1)^2}; & \text{if $n$ is odd}.
\end{cases}$$
\end{theorem}
\begin{proof}~
The proof is similar to that of Theorem \ref{Thm-2} (see Figure \ref{fig:e-sfl} for illustration). 
\end{proof}

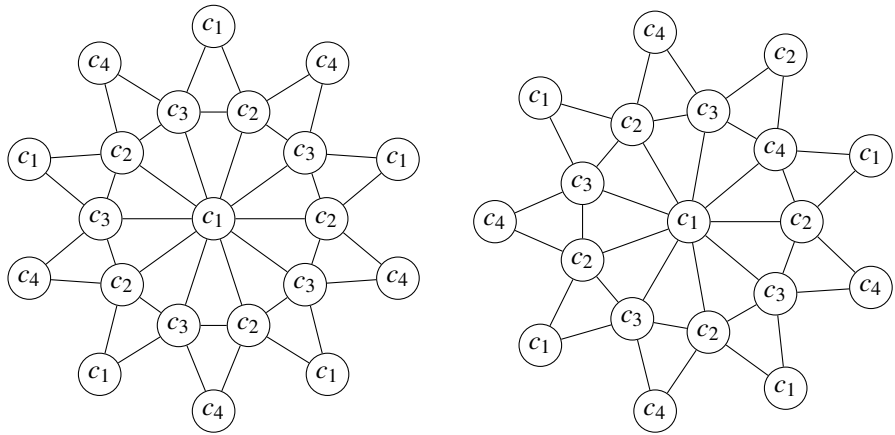
\begin{figure}[h!]
\begin{center}
\begin{tikzpicture}[scale=0.85] 
\vertex (v1) at (0:1.75) []{$c_2$};
\vertex (v2) at (324:1.75) []{$c_3$};
\vertex (v3) at (288:1.75) []{$c_2$};
\vertex (v4) at (252:1.75) []{$c_3$};
\vertex (v5) at (216:1.75) []{$c_2$};
\vertex (v6) at (180:1.75) []{$c_3$};
\vertex (v7) at (144:1.75) []{$c_2$};
\vertex (v8) at (108:1.75) []{$c_3$};
\vertex (v9) at (72:1.75) []{$c_2$};
\vertex (v10) at (36:1.75) []{$c_3$};
\vertex (v) at (0:0) []{$c_{1}$};
\vertex (u1) at (18:3) []{$c_1$};
\vertex (u2) at (342:3) []{$c_4$};
\vertex (u3) at (306:3) []{$c_1$};
\vertex (u4) at (270:3) []{$c_4$};
\vertex (u5) at (234:3) []{$c_1$};
\vertex (u6) at (198:3) []{$c_4$};
\vertex (u7) at (162:3) []{$c_1$};
\vertex (u8) at (126:3) []{$c_4$};
\vertex (u9) at (90:3) []{$c_1$};
\vertex (u10) at (54:3) []{$c_4$};
\path 
(v1) edge (v2)
(v2) edge (v3)
(v3) edge (v4)
(v4) edge (v5)
(v5) edge (v6)
(v6) edge (v7)
(v7) edge (v8)
(v8) edge (v9)
(v9) edge (v10)
(v10) edge (v1)
(v1) edge (v)
(v2) edge (v)
(v3) edge (v)
(v4) edge (v)
(v5) edge (v)
(v6) edge (v)
(v7) edge (v)
(v8) edge (v)
(v9) edge (v)
(v10) edge (v)
(u1) edge (v1)
(u1) edge (v10)
(u2) edge (v1)
(u2) edge (v2)
(u3) edge (v2)
(u3) edge (v3)
(u4) edge (v3)
(u4) edge (v4)
(u5) edge (v4)
(u5) edge (v5)
(u6) edge (v5)
(u6) edge (v6)
(u7) edge (v6)
(u7) edge (v7)
(u8) edge (v7)
(u8) edge (v8)
(u9) edge (v8)
(u9) edge (v9)
(u10) edge (v9)
(u10) edge (v10)
;
\end{tikzpicture}
\qquad 
\begin{tikzpicture}[scale=0.85] 
\vertex (v1) at (0:1.75) []{$c_2$};
\vertex (v2) at (320:1.75) []{$c_3$};
\vertex (v3) at (280:1.75) []{$c_2$};
\vertex (v4) at (240:1.75) []{$c_3$};
\vertex (v5) at (200:1.75) []{$c_2$};
\vertex (v6) at (160:1.75) []{$c_3$};
\vertex (v7) at (120:1.75) []{$c_2$};
\vertex (v8) at (80:1.75) []{$c_3$};
\vertex (v9) at (40:1.75) []{$c_4$};
\vertex (v) at (0:0) []{$c_{1}$};
\vertex (u1) at (20:3) []{$c_1$};
\vertex (u2) at (340:3) []{$c_4$};
\vertex (u3) at (300:3) []{$c_1$};
\vertex (u4) at (260:3) []{$c_4$};
\vertex (u5) at (220:3) []{$c_1$};
\vertex (u6) at (180:3) []{$c_4$};
\vertex (u7) at (140:3) []{$c_1$};
\vertex (u8) at (100:3) []{$c_4$};
\vertex (u9) at (60:3) []{$c_2$};
\path 
(v1) edge (v2)
(v2) edge (v3)
(v3) edge (v4)
(v4) edge (v5)
(v5) edge (v6)
(v6) edge (v7)
(v7) edge (v8)
(v8) edge (v9)
(v9) edge (v1)
(v1) edge (v)
(v2) edge (v)
(v3) edge (v)
(v4) edge (v)
(v5) edge (v)
(v6) edge (v)
(v7) edge (v)
(v8) edge (v)
(v9) edge (v)
(u1) edge (v1)
(u1) edge (v9)
(u2) edge (v1)
(u2) edge (v2)
(u3) edge (v2)
(u3) edge (v3)
(u4) edge (v3)
(u4) edge (v4)
(u5) edge (v4)
(u5) edge (v5)
(u6) edge (v5)
(u6) edge (v6)
(u7) edge (v6)
(u7) edge (v7)
(u8) edge (v7)
(u8) edge (v8)
(u9) edge (v8)
(u9) edge (v9)
;
\end{tikzpicture}
\end{center}
\caption{\small Equitable colouring of a sunflower graphs}\label{fig:e-sfl}
\end{figure} 

A \textit{closed sunflower graph} $CSF_n$ is the graph obtained by joining the independent vertices of a sunflower graph $SF_n$, which are not adjacent to its central vertex so that these vertices induces a cycle on $n$ vertices. The following result provides the equitable colouring parameters of a closed sunflower graph.

 \begin{theorem}\label{Thm-6}
For the sunflower graph $SF_n$, we have $$\eE(CSF_n)=\frac{5n+1}{2n+1}$$ 
and
$$\eV(CSF_n)=\begin{cases}
\frac{5n^2+7n}{(2n+1)^2}; & \text{if $n$ is even}\\
\frac{5n^2+3n-2}{(2n+1)^2}; & \text{if $n$ is odd}.
\end{cases}$$
 \end{theorem}
 \begin{proof}~
The proof is similar to that of Theorem \ref{Thm-2} (see Figure \ref{fig:e-csfl} for illustration). 
 \end{proof}
 
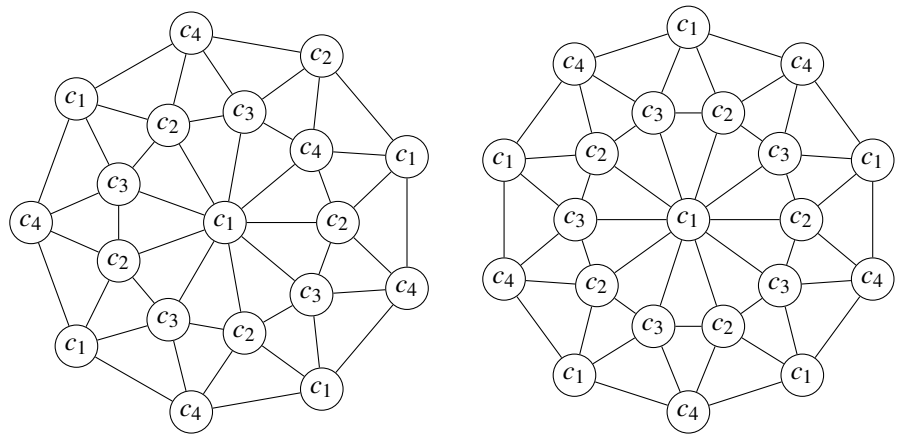
\begin{figure}[h!]
\begin{center}
\begin{tikzpicture}[scale=0.85] 
\vertex (v1) at (0:1.75) []{$c_2$};
\vertex (v2) at (320:1.75) []{$c_3$};
\vertex (v3) at (280:1.75) []{$c_2$};
\vertex (v4) at (240:1.75) []{$c_3$};
\vertex (v5) at (200:1.75) []{$c_2$};
\vertex (v6) at (160:1.75) []{$c_3$};
\vertex (v7) at (120:1.75) []{$c_2$};
\vertex (v8) at (80:1.75) []{$c_3$};
\vertex (v9) at (40:1.75) []{$c_4$};
\vertex (v) at (0:0) []{$c_{1}$};
\vertex (u1) at (20:3) []{$c_1$};
\vertex (u2) at (340:3) []{$c_4$};
\vertex (u3) at (300:3) []{$c_1$};
\vertex (u4) at (260:3) []{$c_4$};
\vertex (u5) at (220:3) []{$c_1$};
\vertex (u6) at (180:3) []{$c_4$};
\vertex (u7) at (140:3) []{$c_1$};
\vertex (u8) at (100:3) []{$c_4$};
\vertex (u9) at (60:3) []{$c_2$};
\path 
(v1) edge (v2)
(v2) edge (v3)
(v3) edge (v4)
(v4) edge (v5)
(v5) edge (v6)
(v6) edge (v7)
(v7) edge (v8)
(v8) edge (v9)
(v9) edge (v1)
(v1) edge (v)
(v2) edge (v)
(v3) edge (v)
(v4) edge (v)
(v5) edge (v)
(v6) edge (v)
(v7) edge (v)
(v8) edge (v)
(v9) edge (v)
(u1) edge (v1)
(u1) edge (v9)
(u2) edge (v1)
(u2) edge (v2)
(u3) edge (v2)
(u3) edge (v3)
(u4) edge (v3)
(u4) edge (v4)
(u5) edge (v4)
(u5) edge (v5)
(u6) edge (v5)
(u6) edge (v6)
(u7) edge (v6)
(u7) edge (v7)
(u8) edge (v7)
(u8) edge (v8)
(u9) edge (v8)
(u9) edge (v9)
(u1) edge (u2)
(u2) edge (u3)
(u3) edge (u4)
(u4) edge (u5)
(u5) edge (u6)
(u6) edge (u7)
(u7) edge (u8)
(u8) edge (u9)
(u9) edge (u1)
;
\end{tikzpicture}
\qquad 
\begin{tikzpicture}[scale=0.85] 
\vertex (v1) at (0:1.75) []{$c_2$};
\vertex (v2) at (324:1.75) []{$c_3$};
\vertex (v3) at (288:1.75) []{$c_2$};
\vertex (v4) at (252:1.75) []{$c_3$};
\vertex (v5) at (216:1.75) []{$c_2$};
\vertex (v6) at (180:1.75) []{$c_3$};
\vertex (v7) at (144:1.75) []{$c_2$};
\vertex (v8) at (108:1.75) []{$c_3$};
\vertex (v9) at (72:1.75) []{$c_2$};
\vertex (v10) at (36:1.75) []{$c_3$};
\vertex (v) at (0:0) []{$c_{1}$};
\vertex (u1) at (18:3) []{$c_1$};
\vertex (u2) at (342:3) []{$c_4$};
\vertex (u3) at (306:3) []{$c_1$};
\vertex (u4) at (270:3) []{$c_4$};
\vertex (u5) at (234:3) []{$c_1$};
\vertex (u6) at (198:3) []{$c_4$};
\vertex (u7) at (162:3) []{$c_1$};
\vertex (u8) at (126:3) []{$c_4$};
\vertex (u9) at (90:3) []{$c_1$};
\vertex (u10) at (54:3) []{$c_4$};
\path 
(v1) edge (v2)
(v2) edge (v3)
(v3) edge (v4)
(v4) edge (v5)
(v5) edge (v6)
(v6) edge (v7)
(v7) edge (v8)
(v8) edge (v9)
(v9) edge (v10)
(v10) edge (v1)
(v1) edge (v)
(v2) edge (v)
(v3) edge (v)
(v4) edge (v)
(v5) edge (v)
(v6) edge (v)
(v7) edge (v)
(v8) edge (v)
(v9) edge (v)
(v10) edge (v)
(u1) edge (u2)
(u2) edge (u3)
(u3) edge (u4)
(u4) edge (u5)
(u5) edge (u6)
(u6) edge (u7)
(u7) edge (u8)
(u8) edge (u9)
(u9) edge (u10)
(u10) edge (u1)
(u1) edge (v1)
(u1) edge (v10)
(u2) edge (v1)
(u2) edge (v2)
(u3) edge (v2)
(u3) edge (v3)
(u4) edge (v3)
(u4) edge (v4)
(u5) edge (v4)
(u5) edge (v5)
(u6) edge (v5)
(u6) edge (v6)
(u7) edge (v6)
(u7) edge (v7)
(u8) edge (v7)
(u8) edge (v8)
(u9) edge (v8)
(u9) edge (v9)
(u10) edge (v9)
(u10) edge (v10)
;
\end{tikzpicture}
\end{center}
\caption{\small Equitable colouring of a closed sunflower graphs}\label{fig:e-csfl}
\end{figure}

A \textit{blossom graph} $Bl_n$ is the graph obtained by joining all vertices of the outer cycle of a closed sunflower graph $CSF_n$ to its central vertex. The equitable chromatic parameters of blossom graphs are determined in the following result.

 \begin{theorem}\label{Thm-7}
For the blossom graph $Bl_n$, we have $$\eE(Bl_n)=\frac{(n+1)^2}{2n+1}$$
and $$\eV(Bl_n)=\frac{n^4+2n^3+2n^2+n}{3(2n+1)^2}.$$
\end{theorem}
\begin{proof}~
The proof is similar to that of Theorem \ref{Thm-4} (see Figure \ref{fig:e-bl} for illustration). 
\end{proof}

\begin{figure}[h!]
\centering
\begin{tikzpicture}[auto,node distance=1.75cm,
thick,main node/.style={circle,draw,font=\sffamily\Large\bfseries}]
\vertex (v1) at (0:1.75) []{$c_1$};
\vertex (v2) at (320:1.75) []{$c_2$};
\vertex (v3) at (280:1.75) []{$c_3$};
\vertex (v4) at (240:1.75) []{$c_4$};
\vertex (v5) at (200:1.75) []{$c_5$};
\vertex (v6) at (160:1.75) []{$c_6$};
\vertex (v7) at (120:1.75) []{$c_7$};
\vertex (v8) at (80:1.75) []{$c_8$};
\vertex (v9) at (40:1.75) []{$c_9$};
\vertex (v) at (0:0) []{$c_{10}$};
\vertex (u1) at (20:3) []{$c_2$};
\vertex (u2) at (340:3) []{$c_3$};
\vertex (u3) at (300:3) []{$c_4$};
\vertex (u4) at (260:3) []{$c_5$};
\vertex (u5) at (220:3) []{$c_6$};
\vertex (u6) at (180:3) []{$c_7$};
\vertex (u7) at (140:3) []{$c_8$};
\vertex (u8) at (100:3) []{$c_9$};
\vertex (u9) at (60:3) []{$c_1$};
\path 
(v1) edge (v2)
(v2) edge (v3)
(v3) edge (v4)
(v4) edge (v5)
(v5) edge (v6)
(v6) edge (v7)
(v7) edge (v8)
(v8) edge (v9)
(v9) edge (v1)
(v1) edge (v)
(v2) edge (v)
(v3) edge (v)
(v4) edge (v)
(v5) edge (v)
(v6) edge (v)
(v7) edge (v)
(v8) edge (v)
(v9) edge (v)
(u1) edge (v1)
(u1) edge (v9)
(u2) edge (v1)
(u2) edge (v2)
(u3) edge (v2)
(u3) edge (v3)
(u4) edge (v3)
(u4) edge (v4)
(u5) edge (v4)
(u5) edge (v5)
(u6) edge (v5)
(u6) edge (v6)
(u7) edge (v6)
(u7) edge (v7)
(u8) edge (v7)
(u8) edge (v8)
(u9) edge (v8)
(u9) edge (v9)
(u1) edge (u2)
(u2) edge (u3)
(u3) edge (u4)
(u4) edge (u5)
(u5) edge (u6)
(u6) edge (u7)
(u7) edge (u8)
(u8) edge (u9)
(u9) edge (u1)
(u1) edge (v)
(u2) edge (v)
(u3) edge (v)
(u4) edge (v)
(u5) edge (v)
(u6) edge (v)
(u7) edge (v)
(u8) edge (v)
(u9) edge (v)
;
\end{tikzpicture}
\caption{\small Equitable colouring of a blossom graphs}\label{fig:e-bl}
\end{figure}
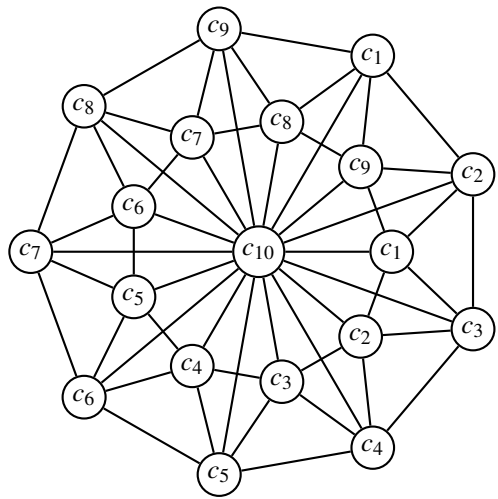

\section{Conclusion}

In this paper, we discussed mean and variance, two important statistical parameters, related to equitable coloring of certain cycle related graphs.  

The coloring parameters can be used in various areas like project management, communication networks, optimisation problems etc. The concepts of equitable chromatic parameters can be utilised in certain practical and industrial problems like resource allocation, resource smoothing, inventory management, service and distribution systems etc.

The $\chi_e$-chromatic mean and variance of many other graph classes are yet to be studied.  Further investigations on the sum, mean and variance corresponding to the $\chi_e$-coloring of many other standard graphs seem to be promising open problems. Studies on the sum, mean and variance corresponding to different types of edge colorings, map colorings, total colorings etc. of graphs also offer much for future studies.

We can associate many other statistical parameters to graph coloring and other notions like covering, matching etc.  All these facts highlight a wide scope for future studies in this area.


\end{document}